\documentclass{amsart}
\usepackage{amssymb}
\usepackage{amsthm}
 
\newtheorem{prop}{Proposition} 
\theoremstyle{definition} 
\newtheorem{ex}{Example}[section]
\theoremstyle{remark} \newtheorem{rem}{Remark}[section]
\newcommand{\pn}{\par\noindent} \newcommand{\pmn}{\par\medskip\noindent}

\begin{document}
\title{An easy proof of Polya's theorem on random walks}
\author{Yury Kochetkov}
\date{}
\begin{abstract} We present an easy proof of Polya's theorem on
random walks: with the probability one a random walk on the
two-dimensional lattice returns to the starting point.
\end{abstract}
\email{yukochetkov@hse.ru, yuyukochetkov@gmail.com} \maketitle

\section{Introduction}
\pn We consider two-dimensional random walks that start at the
origin and performed on the lattice $\mathbb{Z}^2$. There are four
directions for each step and the choice of a direction is random
with probability $1/4$ for each direction. We will consider only
finite walks. A given walk of length $n$ is performed with
probability $1/2^{2n}$. \pmn A \emph{loop} is a walk that begins
and ends at origin. It is convenient to consider the walk of
length zero --- the \emph{trivial} loop. A loop is \emph{simple},
if it is not a concatenation of two nontrivial loops. We will
assume that a simple loop is nontrivial. \pmn Obviously, a loop
has an even length. Let $P_n$ be the number of simple loops of
length $2n$. It is easy to see that $P_1=4$ and $P_2=20$. The
number
$$r=\dfrac{P_1}{2^4}+\dfrac{P_2}{2^8}+\ldots=\sum_{n=1}^\infty
\dfrac{P_n}{2^{4n}}\leqslant 1$$ is the probability of returning
to the origin. \pmn {\bf Polya's theorem on random walk}. $r=1$.
\pmn There are many proofs of this theorem. See, for example,
\cite{DS}, \cite{SSU}, \cite{No}, \cite{La}.

\section{The one-dimensional case}
\pn In our approach the reasoning in the one-dimensional and the
reasoning in the two-dimensional cases are the same, so we will
study the one-dimensional case at first. \pn Let $B_n$ be the
number of loops of the length $2n$, $P_n$ be the number of simple
loops of the same length and
$$B(t)=1+B_1t+B_2t^2+\ldots\quad\text{and}\quad
P(t)=P_1t+P_2t^2+\ldots$$ be the corresponding generating
functions. As each loop is the concatenation of the simple loop
and a loop (maybe trivial), then $B(t)=P(t)\cdot B(t)+1$. Thus,
$P(t)=1-1/B(t)$ in the ring of formal series. \pmn The probability
of a given walk of the length $n$ is $1/2^n$, hence, we will work
with "weighted" generating functions
$$b(x)=B\left(\frac x4\right)=1+\dfrac{B_1}{2^2}\,x+\dfrac{B_2}{2^4}\,x^2+
\ldots\quad\text{and}\quad p\,(x)=P\left(\frac
x4\right)=\dfrac{P_1}{2^2}\,x+ \dfrac{P_2}{2^4}\,x^2+\ldots$$ Then
$$r=p\,(1)=\dfrac{P_1}{2^2}+\dfrac{P_2}{2^4}+\ldots\leqslant 1$$ is
the probability of returning to the origin. So, the convergence
radius of the series $p(x)$ is $\geqslant 1$ \pmn As
$B_n=\binom{2n}{n}$, then the convergence radius of the series
$b(x)$ is one. Indeed,
$$\lim_{n\to \infty} \binom{2n+2}{n+1}\cdot 2^{-(2n+2)}:
\binom{2n}{n}\cdot 2^{-2n}=1.$$ It means that $p(x)=1-1/b(x)$ as
functions, if $0<x<1$. As $\lim_{x\to 1} p(x)=r$, then
$$r=1-\dfrac{1}{\,\,\overset{\vspace{1mm}}{\lim_{x\to 1}
b(x)}}\,.$$ But
$$B_n/2^{2n}=\dfrac{(2n)!}{(n!)^2 2^{2n}} \sim
\dfrac{\sqrt{4\pi n}\cdot 2^{2n}\cdot n^{2n}}{2\pi n\cdot n^{2n}
\cdot 2^{2n}}= \frac{1}{\sqrt{\pi n}}\,.$$ Thus, $\lim_{x\to
1}b(x)=\infty$ and $r=1$.

\section{The two-dimensional case}
\pn As above, let $B_n$ be the number of loops of length $2n$,
$P_n$ be the number of simple loops of the same length,
$$B(t)=1+B_1t+B_2t^2+\ldots\quad\text{and}\quad
P(t)=P_1t+P_2t^2+\ldots$$ be the corresponding generating
functions.  Then
$$r=\dfrac{P_1}{4^2}+\dfrac{P_2}{4^4}+\dfrac{P_3}{4^6}+\ldots
\leqslant 1$$ is the probability of the returning to the origin.
As above we have that $B(t)=P(t)\cdot B(t)+1$ and $P(t)=1-1/B(t)$
in the ring of formal series. We have to compute the number $B_n$.

\begin{prop} $B_n=\binom{2n}{n}^2$.\end{prop}

\begin{proof} Given two strings $a$ and $b$ of $+1$
and $-1$ of length $2n$ , such that there are equal number of plus
and minus ones in each string, we must construct a loop of length
$2n$. It can be done in the following way: the pair $(a[i],b[i])$
corresponds to the $i$-th step
\begin{itemize} \item in the direction $(0,1)$, if
$(a[i],b[i])=(-1,+1)$; \item in the direction $(0,-1)$, if
$(a[i],b[i])=(+1,-1)$; \item in the direction $(1,0)$, if
$(a[i],b[i])=(+1,+1)$; \item and in the direction $(-1,0)$, if
$(a[i],b[i])=(-1,-1)$. \end{itemize} Thus constructed walk returns
to the origin. Indeed, let the number of pairs $(+1,+1)$ is $k$,
the number of pairs $(-1,-1)$ is $l$, the number of pairs
$(+1,-1)$ is $m$ and the number of pairs $(-1,+1)$ is $n$, then
$k+m=l+n$ and $k+n=l+m$. Hence, $2k=2l$, i.e. $k=l$ and $m=n$. So,
the number of "right" steps is equal to the number of "left" steps
and the number of "up" steps is equal to the number of "down"
steps. \end{proof}

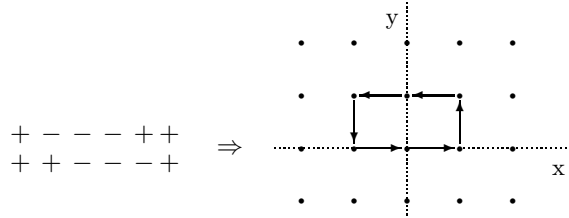
\begin{ex}
\[\begin{picture}(210,90) \put(0,22){$+\,+\,-\,-\,-\,+$}
\put(0,33){$+\,-\,-\,-\,+\,+$} \put(78,28){$\Rightarrow$}
\qbezier[60](100,30)(155,30)(210,30)
\qbezier[50](150,5)(150,45)(150,85) \put(205,20){\small x}
\put(142,77){\small y} \multiput(110,30)(20,0){5}{\circle*{2}}
\multiput(110,10)(20,0){5}{\circle*{2}}
\multiput(110,50)(20,0){5}{\circle*{2}}
\multiput(110,70)(20,0){5}{\circle*{2}}
\put(152,30){\vector(1,0){16}} \put(170,32){\vector(0,1){16}}
\put(168,50){\vector(-1,0){16}} \put(148,50){\vector(-1,0){16}}
\put(130,48){\vector(0,-1){16}} \put(132,30){\vector(1,0){16}}
\end{picture}\] In the opposite way,
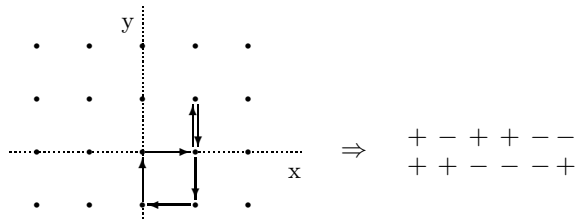
\[\begin{picture}(210,90) \qbezier[60](0,30)(55,30)(110,30)
\qbezier[50](50,5)(50,45)(50,85) \put(105,20){\small x}
\put(42,77){\small y} \multiput(10,30)(20,0){5}{\circle*{2}}
\multiput(10,10)(20,0){5}{\circle*{2}}
\multiput(10,50)(20,0){5}{\circle*{2}}
\multiput(10,70)(20,0){5}{\circle*{2}}
\put(52,30){\vector(1,0){16}} \put(69,32){\vector(0,1){16}}
\put(71,48){\vector(0,-1){16}} \put(70,28){\vector(0,-1){16}}
\put(68,10){\vector(-1,0){16}} \put(50,12){\vector(0,1){16}}
\put(125,28){$\Rightarrow$} \put(150,33){$+\,-\,+\,+\,-\,-$}
\put(150,22){$+\,+\,-\,-\,-\,+$}
\end{picture}\] \end{ex}

\begin{rem} A similar bijection does not exist in the three
dimensional case, because there are eight triads of plus and minus
ones, but only six directions in the three dimensional lattice.
\end{rem}
\pmn As above, let
$$b(x)=B\left(\frac{x}{16}\right)=1+\dfrac{B_1}{4^2}\,x+
\dfrac{B_2}{4^4}\,x^2+ \ldots\quad\text{and}\quad
p\,(x)=P\left(\frac{x}{16}\right)=\dfrac{P_1}{4^2}\,x+
\dfrac{P_2}{4^4}\,x^2+\ldots$$ \pmn \emph{The proof of Polya's
theorem.} As coefficients of the series $b(x)$ are squares of the
coefficients of the same series in the one-dimensional case, then
the radius of convergence of this series is one and its
coefficients are equivalent to $1/\pi n$. Thus this series is
divergent at one and
$$r=\lim_{x\to 1}p(x)=1-1/\lim_{x\to 1} b(x)=1.$$ \qed

\vspace{5mm}

\begin{thebibliography}{99}
\bibitem{DS} P.G. Doyle, J.L. Snell. \emph{Random walks and electrical
networks.} Mathematical Association of America, 1984.
\bibitem{SSU} M. Skopenkov, V. Smykalov, A. Ustinov. \emph{Random walks
and electrical networks.} Matematicheskoe Prosveshzhenie, 16,
2012, 25-47 [in Russian]. \bibitem{No} J. Novak. \emph{Polya's
random walks theorem.} Amer. Math. Monthly, 121(8), 2014.
\bibitem{La} K. Lange. \emph{Polya's random walks theorem revisited.}
Amer. Math. Monthly, 122(10), 2015, 1005-1007.
\end{thebibliography}
\end{document}